\documentclass[a4paper,11pt,abstracton]{scrartcl}

\usepackage[margin=25mm]{geometry}

\usepackage{etex}
\usepackage{scrhack}
\usepackage{cmap}
\usepackage[ngerman,english]{babel}
\usepackage[utf8]{inputenc}
\usepackage{graphicx}
\usepackage{cite}
\usepackage{caption}
\usepackage{subcaption}
\usepackage{epstopdf}
\usepackage{mathtools}
\usepackage{amsmath}
\usepackage{amsthm}
\usepackage{amssymb}
\usepackage{comment}
\usepackage{multirow}
\usepackage{tikz}
\usetikzlibrary{arrows,intersections}
\usetikzlibrary{shapes,snakes}
\usetikzlibrary{graphs,quotes}
\usetikzlibrary{graphs.standard}
\usepackage[english,vlined,ruled]{algorithm2e}
\usepackage{hyperref}
\usepackage{lipsum}
\usepackage{float}
\usepackage{tabularx}
\usepackage{booktabs}
\usepackage{longtable}
\usepackage{tabu}
\usepackage{environ}
\usepackage{mdframed}
\usepackage{datetime}
\usepackage{enumerate}
\usepackage{authblk}

\newtheorem{theorem}{Theorem}%[section]
\newtheorem{corollary}{Corollary} %[section]
\newtheorem{lemma}{Lemma}%[section]
\theoremstyle{definition}
\theoremstyle{remark}

\newenvironment{claim}[1]{\par\medskip\noindent\underline{Claim:}\space#1}{}
\newenvironment{claimproof}[1]{\par\medskip\noindent\underline{Proof of Claim:}\space#1}{\hfill
$\diamond$}

\mdfdefinestyle{probframe}{skipabove=0pt,skipbelow=0pt,innerleftmargin=0pt,
innerrightmargin=5pt,innertopmargin=0pt, innerbottommargin=5pt}

\makeatletter
\newcommand{\problemtitle}[1]{\gdef\@problemtitle{#1}}% Store problem title
\newcommand{\probleminput}[1]{\gdef\@probleminput{#1}}% Store problem input
\newcommand{\problemquestion}[1]{\gdef\@problemquestion{#1}}% Store problem question
\NewEnviron{problem}{
    \medskip
\begin{mdframed}[style=probframe]
  \problemtitle{}\probleminput{}\problemquestion{}% Default input is empty
  \BODY% Parse input
  \par\addvspace{.5\baselineskip}
  \noindent
  \begin{tabularx}{\textwidth}{@{\hspace{\parindent}} l X c}
    \multicolumn{2}{@{\hspace{\parindent}}l}{\@problemtitle} \\% Title
    \textbf{Input:} & \@probleminput \\% Input
    \textbf{Question:} & \@problemquestion% Question
  \end{tabularx}
  \par\addvspace{.5\baselineskip}
\end{mdframed}
    \medskip
}
\makeatother

\newcommand{\nn}{\mathbb{N}}

\newcommand{\X}{\mathcal{X}}
\newcommand{\cU}{\mathcal{U}}

\title{Robust Combinatorial Optimization with Locally Budgeted Uncertainty}

    \author[1]{Marc Goerigk\footnote{marc.goerigk@uni-siegen.de}}
    \author[2]{Stefan Lendl\footnote{lendl@math.tugraz.at}}
	
	\affil[1]{Network and Data Science Management, University of Siegen, Siegen, Germany}	
	\affil[2]{Institute of Discrete Mathematics, Graz University of Technology, Graz, Austria}
	\date{}

\begin{document}

\maketitle

\begin{abstract} 
Budgeted uncertainty sets have been established as a major influence on uncertainty modeling for robust optimization problems. A drawback of such sets is that the budget constraint only restricts the global amount of cost increase that can be distributed by an adversary. Local restrictions, while being important for many applications, cannot be modeled this way.

We introduce new variant of budgeted uncertainty sets, called locally budgeted uncertainty. In this setting, the uncertain parameters become partitioned, such that a classic budgeted uncertainty set applies to each partition, called region.

In a theoretical analysis, we show that the robust counterpart of such problems for a constant number of regions remains solvable in polynomial time, if the underlying nominal problem can be solved in polynomial time as well. If the number of regions is unbounded, we show that the robust selection problem remains solvable in polynomial time, while also providing hardness results for other combinatorial problems.

In computational experiments using both random and real-world data, we show that using locally budgeted uncertainty sets can have considerable advantages over classic budgeted uncertainty sets.
\end{abstract}


\section{Introduction}\label{sec:intro}

% \cite{busing2011phd}, \cite{hradovich2017recoverable} 

We consider nominal combinatorial optimization problems of the form
\begin{align*}
\min\ & \pmb{c}^t \pmb{x} \\
\text{s.t. } & \pmb{x}\in\X
\end{align*}
where $\X\subseteq\{0,1\}^n$ is the set of feasible solutions. For uncertain cost coefficients $\pmb{c}\in\cU$, robust optimization approaches have been analyzed. To this end, one assumes that a set $\cU$ of possible cost realizations is given by a decision maker or derived from historical data. The set $\cU$ is referred to as the uncertainty set. The (min-max) robust counterpart is then to solve
\[ \min_{\pmb{x}\in\X}\ \max_{\pmb{c}\in\cU}\ \pmb{c}^t \pmb{x} \]
Different possibilities to model the set $\cU$ have been proposed. One straight-forward possibility is to use a discrete set of scenarios $\cU = \{ \pmb{c}^1,\ldots,\pmb{c}^N\}$, i.e., to list all possible outcomes explicitly. While this approach is flexible, it usually results in NP-hard robust optimization problems, even if the nominal problem can be solved in polynomial time (see \cite{kouvelis2013robust,aissi2009min,kasperski2016robust} for overviews). Also, implicit descriptions of the uncertainty set can lead to exponential-sized equivalent discrete uncertainty sets.

A popular alternative are budgeted uncertainty sets of the form
\begin{equation}
\cU = \left\{ \pmb{c} = \underline{\pmb{c}} + \pmb{\delta} : \delta_i\in[0,d_i]\ \forall i\in[n],\ \sum_{i\in[n]} \delta_ i \le \Gamma \right\}   \label{classic}
\end{equation}
as first introduced in \cite{bertsimas2003robust,bertsimas2004price}. Here we use the notation $[n]=\{1,\ldots,n\}$. For every item $i\in[n]$, we are given a lower bound on the costs $\underline{c}_i$, as well as a possible maximum cost deviation $d_i$. Additionally, there is a budget $\Gamma$ on the total increase of costs over the lower bound. Advantages of this set include its intuitive
description for a decision maker, and that robust counterparts remain efficiently solvable for nominal problems that can be solved efficiently, even though the budgeted uncertainty set has an exponential number of extreme points. These benefits have lead to a substantial amount of research into robust optimization problems with budgeted uncertainty sets, see, e.g., \cite{alves2015robust,hansknecht2018fast,chassein2018recoverable,bougeret2019robust,chassein2019faster} and many more.

But there are also limitations to this approach, which has lead to the development of alternative uncertainty sets. These include multi-band uncertainty \cite{busing2012new}, variable budgeted uncertainty \cite{poss2013robust}, and knapsack uncertainty \cite{poss2018robust}.

To the best of our knowledge, no previous work has considered avoiding the potential problem that the constraint $\sum_{i\in[n]} \delta_ i \le \Gamma$ denotes a global budget over all uncertain parameters. For various applications, multiple local budgets are more desirable. As examples, consider
 multi-period problems, where every period has its own budget limitation,
routing problems, where separate budgets apply to geographic regions or types of roads,
and portfolio problems, where uncertainty budgets are restricted to asset classes or sectors.

In this paper we introduce a new type of budgeted uncertainty set, where budgets
apply locally to their respective regions. These sets are of the form
\[ \cU = \left\{ \pmb{c} = \underline{\pmb{c}} + \pmb{\delta} : \delta_i \in [0,d_i]\ \forall i\in[n],\ \sum_{i\in P_j} \delta_i \le \Gamma_j\ \forall j\in[K] \right\} \]
where $P_1\cup P_2 \cup \ldots \cup P_K = [n]$ denotes a partition of the items. Each set $P_j$ is called a region. In this approach, every region has a separate budget constraint, which models the
local uncertainty. Note that this definition of uncertainty is a generalization of the classic definition~\eqref{classic}, which can be recovered by using $K=1$.

Our contributions are as follows. For min-max problems with locally budgeted uncertainty, we first derive a compact formulation in Section~\ref{sec:compact}. Based on this formulation, we then consider the case of a constant number of regions in Section~\ref{sec:constant}. We show that the robust problem remains solvable in polynomial time, if it is possible to solve the nominal problem in polynomial time. For an unbounded number of regions, the selection problem remains solvable in polynomial time, while this is not the case for the representative selection problem (see Section~\ref{sec:unbounded}). We conclude that also the spanning tree problem, the $s$-$t$-min-cut problem, and the shortest path problem become NP-hard. Additionally, we can exclude the possibility of parameterized algorithms with running time in $O^*(2^{o(K)})$. In Section~\ref{sec:experiments}, we present three computational experiments using locally budgeted uncertainty sets. In all experiments, we compare locally budgeted uncertainty to the classic budgeted uncertainty approach. While the first two experiments use randomly generated data, the third experiment is based on real-world data for robust shortest path problems. Section~\ref{sec:conclusions} concludes the paper and points out further questions.

\section{Theoretical Results}

% \subsubsection{The Adversarial Problem}

\subsection{A Compact Formulation}
\label{sec:compact}

Let some solution $\pmb{x}\in\X$ be fixed. Its objective value is then determined by solving the adversarial problem
\[ \max_{\pmb{c}\in\cU}\ \pmb{c}^t\pmb{x} \]
that is, by choosing a scenario $\pmb{c}$ that maximizes the costs of $\pmb{x}$. Using the definition of locally budgeted uncertainty, this is equivalent to solving the following linear program:
\begin{align}
\max\ & \sum_{i\in[n]} (\underline{c}_i + \delta_i) x_i \\
\text{s.t. } & \sum_{i\in P_j} \delta_i \le \Gamma_j & \forall j\in[K] \label{adv:con1}\\
& \delta_i \le d_i & \forall i\in[n] \label{adv:con2} \\
& \delta_i \ge 0  & \forall i\in[n]
\end{align}
By strong duality, we can dualize this linear program to find another linear program with the same optimal objective value. Furthermore, any feasible solution to the dual problem gives an upper bound to the objective value of the primal problem. Using the dual, we hence find the following compact problem formulation for the min-max problem with locally budgeted uncertainty.
\begin{align}
\min\ & \sum_{j\in[K]} \left( \Gamma_j \pi_j + \sum_{i\in P_j} d_i \rho_i + \sum_{i\in P_j} \underline{c}_i x_i \right) \label{compactstart}\\
\text{s.t. } & \pi_j + \rho_i \ge x_i & \forall j\in[K], i\in P_j \\
& \pi_j \ge 0 & \forall j\in [K] \\
& \rho_i \ge 0 & \forall i\in[n] \\
& \pmb{x} \in \X \label{compactend}
\end{align}
Recall that $\X$ represents the set of feasible solutions for the underlying combinatorial problem. Variables $\pi_j$ are the duals of Constraints~\eqref{adv:con1}, and variables $\rho_i$ are the duals of Constraints~\eqref{adv:con2}.

\subsection{Constant Number of Regions}
\label{sec:constant}

We first consider the case that the number of regions $K$ is a constant value.
% \subsubsection{A Fixed Parameter Tractable Meta-Algorithm}
% We reconsider the compact formulation. 
Note that, in an optimal solution, we can assume that $\rho_i = [x_i - \pi_j]_+$, where $[y]_+=\max\{0,y\}$ denotes the positive part of $y$. 

\begin{lemma}\label{pilemma}
There is an optimal solution to Problem~(\ref{compactstart}-\ref{compactend}), where $\pi_j \in \{0,1\}$ for all $j\in[K]$.
\end{lemma}
\begin{proof}
Let us assume that $\pmb{x}\in\X$ is fixed. Let $X = \{ i\in[n] : x_i = 1\}$ denote the set of items taken by solution $\pmb{x}$. The problem then decomposes to:
\begin{align}
& \min_{\pmb{\pi} \ge \pmb{0}} \sum_{j\in[K]} \left( \Gamma_j \pi_j + \sum_{i\in P_j} d_i \rho_i + \sum_{i\in P_j} \underline{c}_i x_i \right) \\
= & \sum_{i\in[n]} \underline{c}_i x_i + 
\sum_{j\in[K]} \min_{\pi_j \ge 0} \left( \Gamma_j\pi_j + \sum_{i\in P_j} d_i[x_i - \pi_j]_+ \right) \\
= & \sum_{i\in X} \underline{c}_i 
+ \sum_{j\in[K]} \min_{\pi_j \in[0,1]} \left( \Gamma_j\pi_j + \sum_{i\in P_j\cap X} d_i(1 - \pi_j) \right) \label{lemma1con}
\end{align}
Note that the equivalence~\eqref{lemma1con} follows as increasing any variable $\pi_j$ beyond 1 can never be optimal for $\Gamma_j > 0$. If $\Gamma_j=0$, then setting $\pi_j=1$ gives the same value as setting $\pi_j > 1$.
We can conclude that there is an optimal solution with $\pi_j \in \{0,1\}$ for all $j\in[K]$. 
\end{proof}

\begin{theorem}\label{decomp-theorem}
The robust problem with locally budgeted uncertainty (\ref{compactstart}-\ref{compactend}) can be decomposed into $2^K$ subproblems of nominal type. In particular, if $K$ is a constant and the nominal problem can be solved in polynomial time, Problem~(\ref{compactstart}-\ref{compactend}) can be solved in polynomial time as well.
\end{theorem}
\begin{proof}
By Lemma~\ref{pilemma}, we can assume every variable $\pi_j$ to be either 0 or 1. We guess these values. There are $K$ variables $\pi_j$, and thus $2^K$ combinations are possible. For fixed $\pmb{\pi}=(\pi_1,\ldots,\pi_K)$, denote by $\Pi\subseteq[K]$ the set of indices $j$ where $\pi_j=1$. The problem then becomes
\begin{align*}
& \min_{\pmb{x}\in\X} \sum_{j\in[K]} \left( \Gamma_j \pi_j + \sum_{i\in P_j} d_i [x_i - \pi_j]_+ + \sum_{i\in P_j} \underline{c}_i x_i \right) \\
= & \sum_{j\in \Pi} \Gamma_j 
+ \min_{\pmb{x}\in\X} \left( \sum_{j\in \Pi} \sum_{i\in P_j} \underline{c}_i x_i + \sum_{j\in [K]\setminus \Pi}  \sum_{i\in P_j} (\underline{c}_i + d_i) x_i \right)
\end{align*}
This is a problem of nominal type, and the claim follows.
\end{proof}

% \subsubsection{Improving the FPT constant for the representative selection problem}

\subsection{Unbounded Number of Regions}
\label{sec:unbounded}

We now consider the case that the number of regions $K$ is not a constant, but part of the problem input.

\subsubsection{Hardness Results}

We first consider the representative selection problem, where
\[ \X = \left\{ \pmb{x}\in\{0,1\}^n : \sum_{i\in T_\ell} x_i = p_\ell\ \forall \ell\in [L] \right\} \]
for a partition $T_1\cup T_2 \cup \ldots T_L = [n]$ and integers $p_\ell$ for all $\ell\in[L]$ (see, e.g., \cite{deineko2013complexity}).

\begin{theorem}\label{thm:nphard}
The robust representative selection problem with locally budgeted uncertainty and arbitrary $K$ is APX-hard, 
even if $|T_\ell| = 2$, $p_\ell=1$ for all $\ell\in[L]$, and the for the regions it holds
that $|P_j| \leq 3$ for all $j \in [K]$.
\end{theorem}
\begin{proof}
We reduce from an instance of the vertex cover problem, which is
APX-hard, even on 3-regular graphs \cite{garey1979computers,alimonti2000some}.

\noindent\textbf{Given:} Graph $G=(V,E)$ 3-regular, $k \in \nn$

\noindent\textbf{Question:} Does there exist a vertex cover of size less or equal to $k$, i.e. a set $S \subseteq V$ such that for all
            $e=\{u,v\} \in E$ it holds that $u \in S$ or $v \in S$, and $|S| \leq k$?

    \begin{figure}[ht]
        \centering
        \begin{tikzpicture}
          %\graph[clockwise, radius=3cm] {
          %  subgraph C_n [n=6,name=A];
          %};

          \graph[clockwise=6, radius=2.3cm] {
              {1 [ultra thick, circle, draw], 2, 3[ultra thick, circle, draw], 4, 5[ultra thick, circle, draw], 6[ultra thick, circle, draw]};

              {1}  --[edge label=a] {2};
              {2}  --[edge label=b] {3};
              {3}  --[edge label=c] {4};
              {4}  --[edge label=d] {5};
              {5}  --[edge label=e] {6};
              {6}  --[edge label=f] {1};
              {1}  --["g", pos=0.2] {4};
              {2}  --["h", pos=0.8] {5};
              {3}  --["i", pos=0.2] {6};
          };

          %\foreach \i [evaluate={\j=int(\i+3)}] in {1,2,3}{
          %    \draw (A \i) -- (A \j);
          %}

          %\draw [red, very thick] (A 1) circle [radius=3.4mm];
          %\draw [red, very thick] (A 2) circle [radius=3.4mm];
          %\draw [red, very thick] (A 3) circle [radius=3.4mm];
          %\draw [red, very thick] (A 5) circle [radius=3.4mm];

          %\foreach \i [evaluate={\j=int(mod(\i+2+4,5)+1)}]% using Paul Gaborit's optimisation
          %   in {1,2,3,4,5}{
          %  \draw (A \i) -- (B \i);
          %  \draw (B \j) -- (B \i);
          %}
        \end{tikzpicture}

        \bigskip

        \begin{tikzpicture}[box/.style={rectangle,draw=black,thick, minimum
            size=0.8cm},]
            \node[box,fill=red!80] at (0,-0.8){$2_a$};
            \node[box,fill=green,line width=3pt] at (0,0){$\mathbf{1_a}$};
            \node[box,fill=red!80] at (1,0){$2_b$};
            \node[box,fill=blue!70,line width=3pt] at (1,-0.8){$\mathbf{3_b}$};
            \node[box,fill=orange] at (2,-0.8){$4_c$};
            \node[box,fill=blue!70,line width=3pt] at (2,0){$\mathbf{3_c}$};
            \node[box,fill=orange] at (3,0){$4_d$};
            \node[box,fill=yellow,line width=3pt] at (3,-0.8){$\mathbf{5_d}$};
            \node[box,fill=pink] at (4,-0.8){$6_e$};
            \node[box,fill=yellow,line width=3pt] at (4,0){$\mathbf{5_e}$};
            \node[box,fill=green] at (5,-0.8){$1_f$};
            \node[box,fill=pink,line width=3pt] at (5,0){$\mathbf{6_f}$};
            \node[box,fill=orange] at (6,-0.8){$4_g$};
            \node[box,fill=green,line width=3pt] at (6,0){$\mathbf{1_g}$};
            \node[box,fill=red!80] at (7,0){$2_h$};
            \node[box,fill=yellow,line width=3pt] at (7,-0.8){$\mathbf{5_h}$};
            \node[box,fill=blue!70] at (8,0){$3_i$};
            \node[box,fill=pink,line width=3pt] at (8,-0.8){$\mathbf{6_i}$};

        \end{tikzpicture}
        \caption{Illustration of the reduction from vertex cover. The big
        circled vertices correspond to a minimum size vertex cover of the graph.
    Below we show the instance of the robust representative selection problem
corresponding to this graph. Each column corresponds to a partition from which
one of the two elements must be selected. The colors correspond to the regions
of the instance. The bold elements are an optimal solution corresponding to the
 shown vertex cover of the graph. Note that only elements in 4 regions (green,
 blue, yellow, pink), corresponding to the vertices 1,3,5,6, are selected.}
    \label{fig:vcred}
    \end{figure}
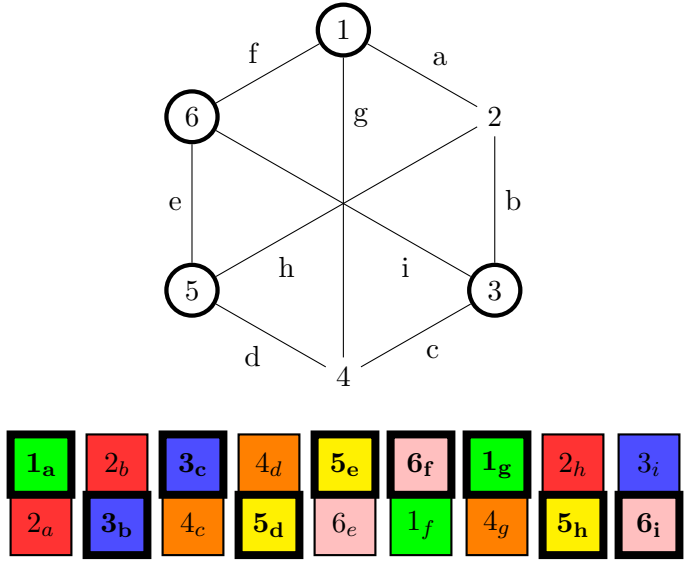

    Given such an instance, we construct an instance of the robust representative
    selection problem with locally budgeted uncertainty fulfilling the
    restrictions stated in the theorem. In Figure~\ref{fig:vcred} we illustrate
    the reduction via an
    example for a concrete vertex cover instance.
    Let $L = |E|$ and $n=2|E|$. 
    For each $e = \{u, v\} \in E$ let $u_e,
    v_e$ be the two elements of $[n]$ in $T_e$, which we associate with the two
    vertices $u$ and $v$. Note that for every vertex $v\in V$ there exist degree of $v$ many
    elements that are associated with this vertex, one corresponding to every edge
    incident to $v$.
    
    For our partition into regions of the locally budgeted uncertainty set, we use exactly those sets of elements
    that correspond to the same vertices in $V$, i.e. we define a region
    $P_v = \{ v_e \colon e \in E, v \text{ incident to }e \}$ for every $v \in
    V$. 
    Note that these sets also form a partition of $[n]$, and we have $K=|V|$. We further set $\pmb{\Gamma} = \pmb{1}$, $\underline{\pmb{c}} =
    \pmb{0}$ and $\pmb{d} = \pmb{1}$.

    We show that there is a vertex cover of size at most $k$ if and only if the
    constructed instance of the robust representative selection problem with
    locally budgeted uncertainty has a solution with objective value at most $k$.
    To see this we first prove the following claim.
    \begin{claim}
        Given a feasible solution $\pmb{x}$ of our instance of the robust representative
        selection problem, the robust objective value is equal to the number of
        regions $P_v$ in which at least one element is selected, i.e.
        \[ | \{ v \in V \colon \exists i \in P_v \text{ with } x_i = 1 \}|. \]
%         \[ | \{ v \in V \colon \exists i \in [n], x_i = 1, i \in P_v \}|. \]
    \end{claim}
    \begin{claimproof}[Proof of Claim]
        First observe that this value can be realized by the adversary by
        selecting for each region $P_v$ with an element $i \in P_v$ such that 
        $x_i=1$ an arbitrary such element $i$ and set $\delta_i = 1$. It
        is easy to see that this is a feasible solution for the adversary and
        the claimed objective value is reached.

        For the upper bound observe that it is only sensible to set 
        $\delta_i > 0$ for $i\in [n]$ with $x_i = 1$. Hence by
        the definition of $\cU$ we have that
%         $ | \{ v \in V \colon \exists i \in [n], \pmb{x}_i = 1, i \in P_v \}| $
        $ | \{ v \in V \colon \exists i \in P_v \text{ with } x_i = 1 \}| $
        is an upper bound on the objective value of the adversary.
    \end{claimproof}

    Now given a vertex cover $S$ of size $k$ we construct a solution $\pmb{x}$ by
    selecting in each $T_e$ the element corresponding to the vertex in the
    vertex cover. If both vertices incident to $e$ are in $S$ we choose one of
    the two elements arbitrarily. Since $|S|=k$ we select elements from at most
    $k$ different regions $P_v$. Hence, by our claim the objective value of the
    robust representative selection problem is less or equal to $k$.

    Given a solution $\pmb{x}$ to the robust representative selection problem
    with objective value $k$, we know by our claim that the elements selected by
    $\pmb{x}$ are contained in exactly $k$ different regions $P_{v_1}, \dots,
    P_{v_k}$.
    We define $S$ to be the set corresponding
    to exactly those $k$ different vertices, i.e. $S = \{v_{1}, \dots, v_{k}\}$. 
    Since for each $T_e, e \in E$ one
    element is selected by $\pmb{x}$, also for every edge $e$ at least one
    incident vertex is contained in $S$, hence $S$ is a vertex cover of size
    $k$.
\end{proof}

\begin{corollary}
The robust problem with locally budgeted uncertainty with arbitrary $K$ is
APX-hard for the shortest path problem on series-parallel graphs, for the minimum
spanning tree problem, and for the $s$-$t$-min-cut problem, even if  
for the regions it holds that $|P_j| \leq 3$ for all $j \in [K]$.
\end{corollary}
\begin{proof}
    The result for the shortest path and minimum spanning tree problem follows
    directly from Theorem~\ref{thm:nphard}. To see this, given an instance of
    the representative selection problem with $|T_\ell| = 2$ and $p_\ell=1$ for all $\ell
    \in [L]$, we construct the graph $G$ with vertex set $V = \{0,1,\dots,L\}$
    and edge set $E$ consisting of parallel edges $e_\ell^1, e_\ell^2$ connecting vertex $\ell-1$ with vertex $\ell$ for all $\ell \in [L]$. Here $e_\ell^1, e_\ell^2$ are in one-to-one correspondence with the two
    elements in $T_\ell$. It is now easy to see that both spanning trees and paths
    from nodes $0$ to $L$ in $G$ are in one-to-one correspondence with feasible
    solutions to the original representative selection problem.
    Using the same locally budgeted uncertainty set as for
    the representative selection problem, we find that objective values
    of corresponding solutions remain equal.

    To obtain the result for the $s$-$t$-min-cut problem, observe that the special instance
    of the representative selection problem is equivalent to the $s$-$t$-cut problem in a graph $G$
    where for each part $\ell \in [L]$ of size $2$ we add a special vertex $v_{\ell}$ in addition to 
    $s$ and $t$ and the path from $s$ via $v_{\ell}$ to $t$. Then $s$-$t$-cuts correspond to selecting 
    one of the two edges from each of these paths.
\end{proof}

Note that the above reduction does not exclude the possibility of a
parameterized algorithm with running time $O^{*}(2^{o(K)})$, even if we assume 
the exponential time hypothesis (ETH), since
the number of regions $K$ in the reduction cannot be bounded by the solution
size $k$ of the vertex cover.
In the following we give a direct linear parameterized reduction from 3-SAT
to robust representative selection with locally budgeted uncertainty, which
shows that the running time of our FPT meta-algorithm is essentially tight under
ETH.
\begin{theorem}\label{thm:ethhard}
    Assuming ETH, there is no $O^{*}(2^{o(K)})$ time algorithm for the robust
    representative selection problem with locally budgeted uncertainty, even if
    $|T_\ell| \leq 3$ and $p_\ell=1$ for all $\ell\in[L]$.
\end{theorem}
\begin{proof}
    We reduce from an instance of the well known 3-SAT problem.
    
\noindent\textbf{Given:} A formula $\varphi$ in 3-CNF with variables $x_1, \dots,
            x_{\tilde{n}}$ and $\tilde{m}$ clauses, i.e.
            \[ \varphi = (l_{1,1} \vee l_{1,2} \vee l_{1,3}) \wedge \dots \wedge (l_{\tilde{m},1} \vee
            l_{\tilde{m},2} \vee l_{\tilde{m},3}), \]
            where the $l$ are literals of the variables $\pmb{x}$.
        
\noindent\textbf{Question:} Is there an assignment for $\pmb{x}$ such that $\varphi$ is
            true?

    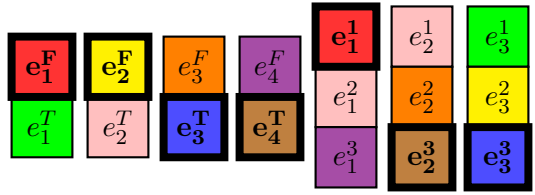
\begin{figure}[ht]
        \centering
        \begin{tikzpicture}
            \node at (0,0) {$\varphi = (\bar{x}_{1} \vee x_2 \vee \bar{x}_4) \wedge (x_2 \vee \bar{x}_3 \vee x_4) \wedge (x_1 \vee \bar{x}_2 \vee x_3)$};

            \node at (0, -1) {$x_1 = 0, x_2=0, x_3=1, x_4=1$};
        \end{tikzpicture}

        \bigskip

        \begin{tikzpicture}[box/.style={rectangle,draw=black,thick, minimum
            size=0.8cm},]

            \node[box,fill=green] at (0,-0.8){$e_1^T$};
            \node[box,fill=red!80,line width=3pt] at (0,0){$\mathbf{e_1^F}$};

            \node[box,fill=pink] at (1,-0.8){$e_2^T$};
            \node[box,fill=yellow,line width=3pt] at (1,0){$\mathbf{e_2^F}$};

            \node[box,fill=orange] at (2,0){$e_3^F$};
            \node[box,fill=blue!70,line width=3pt] at (2,-0.8){$\mathbf{e_3^T}$};

            \node[box,fill=violet!70] at (3,0){$e_4^F$};
            \node[box,fill=brown,line width=3pt] at (3,-0.8){$\mathbf{e_4^T}$};

            \node[box,fill=pink] at (4,-0.4){$e_1^2$};
            \node[box,fill=violet!70] at (4,-1.2){$e_1^3$};
            \node[box,fill=red!80,line width=3pt] at (4,0.4){$\mathbf{e_1^1}$};

            \node[box,fill=pink] at (5,0.4){$e_2^1$};
            \node[box,fill=orange] at (5,-0.4){$e_2^2$};
            \node[box,fill=brown,line width=3pt] at (5,-1.2){$\mathbf{e_2^3}$};

            \node[box,fill=green] at (6,0.4){$e_3^1$};
            \node[box,fill=yellow] at (6,-0.4){$e_3^2$};
            \node[box,fill=blue!70,line width=3pt] at (6,-1.2){$\mathbf{e_3^3}$};

        \end{tikzpicture}
        \caption{Illustration of the reduction from 3-SAT. Above there is an example of a 3-SAT formula with a feasible assignment.
    Below we show the instance of the robust representative selection problem
corresponding to this instance. Each column corresponds to a partition from which
one of the two elements must be selected. The colors correspond to the regions
of the instance. The bold elements are an optimal solution corresponding to the
 shown variable assignment. Note that elements in exactly 4 regions (red,
 yellow, blue, brown), corresponding to the shown feasible variable assignment, are selected.}
    \label{fig:satred}
    \end{figure}

    The exponential time hypothesis (ETH) implies that there does not exist an algorithm to decide 3-SAT
    with running time $2^{o(n)}$, and is widely believed~\cite{woeginger2003exact}.
    We define an instance of our robust problem on the element set $[2 \tilde{n}
    + 3 \tilde{m}]$ in the following way. In Figure~\ref{fig:satred} we illustrate
    the reduction via an
    example for a 3-SAT instance. The 
    partition of $[2\tilde{n}+3\tilde{m}]$ for the representative selection problem consists of $\tilde{n} +
    \tilde{m}$ parts, one for each variable and clause. For each $i \in
    [\tilde{n}]$ the set $T_{i}$ consists of two elements, an element $e^{T}_{i}$  and an element
    $e^{F}_{i}$. These parts are the variable gadgets and selecting $e^{T}_{i}$
    or $e^{F}_{i}$ corresponds to setting $x_{i}$ to true or false respectively.
    For each clause we create a part consisting of exactly three elements, i.e.
    for each $j \in [\tilde{m}]$ the set $T_{\tilde{n}+j}$ consists of the
    elements $e^{1}_{j}$, $e^{2}_{j}$ and $e^{3}_{j}$.
    
    Using the locally
    budgeted uncertainty set and cost structure we will enforce that $e^{i}_{j}$
    can only be selected without inducing additional cost, 
    if the selection in the variable gadget corresponding
    to the variable of literal $l_{j,i}$ corresponds to $l_{j,i}$ being true.
    To this aim, we define the partition for the locally budgeted uncertainty
    set consisting of $K = 2\tilde{n}$ regions $P_{i}^{T}$ and $P_{i}^{F}$ for each $i
    \in [\tilde{n}]$. Selecting an element inside region $P_{i}^{T}$ or
    $P_{i}^{F}$ corresponds to setting the variable $x_i$ to true or false
    respectively.
    We set $\Gamma_{i}^{T} = \Gamma_{i}^{F} = 1$ for all $i \in [\tilde{n}]$ and the
    costs to $\underline{\pmb{c}} = \pmb{0}$ and $\pmb{d} = \pmb{1}$. In a
    similar way as in the proof of Theorem~\ref{thm:nphard}, one can prove the following claim.
    \begin{claim}
        Given a feasible solution $\pmb{x}$ of our instance of the robust representative
        selection problem, the robust objective value is equal to the number of
        regions in which at least one element is selected by $\pmb{x}$.
    \end{claim}

    Based on this, we show that $\varphi$ is feasible, if and only if the
    objective value of our instance is $\tilde{n}$.

    Given a feasible assignment for $\varphi$ we select in each variable gadget
    the corresponding element. Then, since $\varphi$ is true, for each clause $j
    \in [\tilde{m}]$ there is at least one literal $l_{j,i}$ which is true. We
    select the corresponding element $e^{i}_{j}$ in $T_{\tilde{n}+j}$. Observe
    that this selection uses exactly the $\tilde{n}$ parts $P_{i}^{v}$ where $v$
    is $T$ if $x_i$ is true and $v$ is $F$ if $x_i$ is false.  Hence by our
    claim the objective value of this selection is $\tilde{n}$.

    For the other direction first observe, that the objective value of our
    instance cannot be smaller than $\tilde{n}$, since in every variable gadget
    one element must be chosen and each element $e_{i}^{v}$ has its exclusive
    region $P_{i}^{v}$. Now assume that there is a selection with
    robust objective value $\tilde{n}$. Then by our claim for every $i \in
    [\tilde{n}]$ in exactly one of the two regions $P_{i}^{T}$ and $P_{i}^{F}$
    an element of $T_{i}$ is selected. Hence, the truth assignment to $x_{i}$
    induced by the selection in the variable gadget satisfies all the
    clauses, since otherwise in one of the clause gadgets we would have to
    select an element inside an additional region.
\end{proof}

This result is slightly weaker than Theorem~\ref{thm:nphard} in the
sense that we cannot assume $|T_\ell| = 2$ but only $|T_\ell| \leq 3$. The existence
of a $O^{*}(2^{o(K)})$ time algorithm for this case is an open problem.

\begin{theorem}\label{theolog}
Let $k=\max_{\ell\in[L]} |T_\ell|$ and $\Delta=\max_{j\in[K]} |P_j|$. Then the robust representative selection problem with locally budgeted uncertainty and arbitrary $K$ inherits inapproximability results from the set cover problem with maximum cardinality of subsets $\Delta$, and maximum number $k$ of subsets containing any element of the ground set. 
\end{theorem}
\begin{proof}
We use an objective-preserving reduction from the set cover problem.

\noindent\textbf{Given:} A ground set $V$ with $|V|=\tilde{n}$, and a set of subsets $V_s \subseteq V$, $s\in S$ with $|S|=\tilde{m}$.

\noindent\textbf{Question:} Does there exist a set cover of size less or equal to $k$, i.e., a set $C\subseteq S$ with $|C|\le k$ such that $\cup_{s\in C} V_s = V$?

We set $L=\tilde{n}$ and $K=\tilde{m}$. For each $v\in V$ and each $s\in S$, we add an element $(v,s)$ to set $T_v$. Each such element $(v,s)$ belongs to a region $P_s$. We set $\Gamma_s = 1$ for all $s \in S$ and the
costs to $\underline{\pmb{c}} = \pmb{0}$ and $\pmb{d} = \pmb{1}$. Finally, we set $p_v = 1$ for all sets $T_v$.

As an example, let us assume we have $\tilde{n}=4$, $\tilde{m}=3$, and $V_1 = \{1,2\}$, $V_2=\{2,3\}$, $V_3 = \{3,4\}$. Then Table~\ref{tablelog} illustrates the construction.
\begin{table}[h]
\begin{center}
\begin{tabular}{r|rrr|r}
 & $V_1$ & $V_2$ & $V_3$ & \\
\hline
1 & x &   &   & $T_1$ \\
2 & x & x &   & $T_2$ \\
3 &   & x & x & $T_3$ \\
4 &   &   & x & $T_4$ \\
\hline
  &$P_1$ & $P_2$ & $P_3$
\end{tabular}
\caption{Example construction for the proof of Theorem~\ref{theolog}.}\label{tablelog}
\end{center}
\end{table}
By choosing one item from each set $T_v$, we determine by which set $V_s$ we intend to cover it. It can be easily seen that a set cover of size $k$ exists if and only if the robust representative selection problem with locally budgeted uncertainty has an objective value at most $k$. Hence, the reduction is cost-preserving, and the claim follows.
% Note that the size of each region $P_s$ is at most $\tilde{n}$. Hence, if it was possible to approximate the robust representative selection problem in some factor $O(\alpha(n^{max}))$ with $n^{max} = \max_{j \in [K]} |P_j|$, then it is also possible to approximate the set cover problem within $O(\alpha(\tilde{n}))$. As set cover cannot be approximated in $\Omega(\log \tilde{n})$, 
\end{proof}

Note that the inapproximability of set cover under parameters $\Delta$ and $k$ is well-researched (see, e.g., \cite{saket2012new}). If $k=\theta(\log\Delta)$, both problems become hard to approximate within $\Omega(\log\Delta/(\log\log\Delta)^2)$.

\subsubsection{A Polynomial Time Algorithm for the Selection Problem}
\label{sec:selection}

While the results from the previous section indicate that the robust counterpart of even simple combinatorial problems becomes hard, we now show that this is not the case for the selection problem, where $\X = \{ \pmb{x}\in\{0,1\}^n : \sum_{i\in[n]} = p\}$ for some integer $p$.

To this end, use a dynamic program over how many items are taken from every partition. Let $n_j=|P_j|$ be the number of items in region $h\in[K]$. If the number of items $p_j$ taken from partition $P_j$ is fixed, the robust problem can be decomposed:
\[
\min_{\pmb{x}\in X} \max_{\pmb{c}\in\cU} \pmb{c}^t\pmb{x} 
= \sum_{j\in [K]} \min_{\pmb{x}\in\X_j} \max_{\pmb{c}\in\cU_j} \pmb{c}^t\pmb{x} 
=: \sum_{j\in[K]} f_j(p_j)
\]
where $\X_j = \{ \pmb{x}\in\{0,1\}^{n_j} : \sum_{i\in P_j} x_i = p_j\}$ and
$\cU_j = \{ \pmb{c} \in\mathbb{R}^{n_j} : c_i = \underline{c}_i + \delta_i,\
\delta_i \in [0,d_i]\ \forall i\in P_j,\ \sum_{i\in P_j} \delta_i \le
\Gamma_j\}$.

For every $j\in[K]$ and every $p_j\in\{0,1,\ldots,n_j\}$, 
the value $f_j(p_j)$ is the solution of a robust selection problem with continuous budgeted uncertainty set.
In Theorem~\ref{thm:polysel} we explain how the whole table of these values can be precomputed efficiently.

The robust selection problem with locally budgeted uncertainty thus becomes
\begin{align}
\min\ &\sum_{j\in[K]} f_j(p_j) \label{dpstart}\\
\text{s.t. } & \sum_{j\in[K]} p_j = p \\
& p_j \in\mathbb{N}_0 & \forall j\in[K]  \label{dpend}
\end{align}
where we use $f_j(p_j) = \infty$ if $p_j > n_j$

\begin{theorem}\label{thm:polysel}
The robust selection problem with locally budgeted uncertainty with arbitrary
number of regions $K$ can be solved in $O(n \log n + pn)$ time, hence in polynomial time.
\end{theorem}
\begin{proof}
First we explain how to compute the complete table of values $f_j(p_j)$
for all values $j=1,\dots,K$ and $p_j=0,\dots,p$.
Note that for fixed $j$ and $p_j$ computing
$f_j(p_j)$ is equivalent to solving the robust selection problem with continuous budgeted uncertainty.
Observe that this problem is equivalent to solving
\[ \min_{\pmb{x} \in \X_j} \left(\underline{\pmb{c}}^t\pmb{x} + \min\{ \pmb{d}^t\pmb{x}, \Gamma_j \}\right),  \] 
corresponding to the two cases of $\pi_j=0$ and $\pi_j=1$ in formulation~(\ref{compactstart}-\ref{compactend}).
The optimal solution to this problem can be determined by solving the two instances of 
the selection problem for parameter $p_j$ with costs $\underline{\pmb{c}}$ and 
$(\underline{\pmb{c}} + \pmb{d})$ and taking the solution giving the smaller objective value. Hence, for fixed $j$, all values of $f_j$ can be calculated by sorting the items in $P_j$ once with respect to costs $\underline{\pmb{c}}$, and once with respect to costs $\underline{\pmb{c}}+\pmb{d}$. In total, this requires time $O(\sum_{j\in[K]} n_j \log n_j) = O(\sum_{j\in[K]} n_j \log n) = O(n \log n)$.

We now give a dynamic program solving problems of type~(\ref{dpstart}-\ref{dpend}) in general form,
based on a given table for the values of values $f_j(p_j)$. This
then directly implies our result for the robust selection problem.
Let $T(K',p')$ be defined as
\begin{align*}
T(K',p') := \min\ &\sum_{j\in[K']} f_j(p_j) \\
\text{s.t. } & \sum_{j\in[K']} p_j = p' \\
& p_j \in\mathbb{N}_0 & \forall j\in[K'].
\end{align*}
Then problem~(\ref{dpstart}-\ref{dpend}) is equivalent to computing $T(K, p)$. It holds that $T(1,p') =
f_1(p')$ for all $p'=0,1,\dots,p$. It also holds that
\[ T(K', p') = \min \left\{ T(K'-1, p'-p_{K'}) + f_{K'}(p_{K'}) \colon
p_{K'}=0,1, \dots, \min\{p',n_j\} \right \}. \]
Hence, calculating entry $T(K',p')$ can be done in $O(n_j)$ time, if all preceding entries have already been calculated. In total, this means that $T(K,p)$ can be calculated in $O(\sum_{j\in[K]} \sum_{p'\in[p]} n_j) = O(\sum_{j\in[K]} p n_j) = O(pn)$ time.

Hence in total the running time of our algorithm is $O(n \log n + p n)$.
\end{proof}

\section{Experiments}
\label{sec:experiments}

\subsection{Overview}

We present three experiments to quantify differences between ''classic'' budgeted uncertainty sets and the locally budgeted uncertainty sets proposed in this paper. Experiments 1 and 2 use randomly generated data for the uncertain selection problem, while Experiment 3 is based on real-world data.
% \begin{itemize}
% \item 
In the first experiment, we assume that the uncertainty set is locally budgeted, and consider the benefit of using this information instead of using a classic budgeted set.
% \item 
In the second experiment, the actual regions are not known to the imagined decision maker. Instead, only sampled scenarios are provided. We analyze the differences between solutions based on classic and locally budgeted uncertainty sets fitted to the data.
% \item 
Finally, in the third experiment, we consider the differences between solutions based on classic and locally budgeted uncertainty sets fitted to real-world data for shortest path problems, where nothing is known about the underlying distribution.
% \end{itemize}

% computer environment not relevant for these experiments

\subsection{Experiment 1}

\subsubsection{Setup}

% We generate random selection problems in the following way. 
In this experiment, we focus on randomly generated selection problems.
We fix $n=30$. Given the number of regions $K$, we distribute items into the $K$ regions as uniformly as possible. 
For every item, we generate $\underline{c}_i$ and $d_i$ independently and uniformly from $\{10,\ldots,49\}$. We set $\Gamma_j = 10|P_j|$ and use $K=2,3,4,5$. We generate 10,000 instances using the same random seed for each $K$, (i.e., cost coefficients of these instances are the same for each $K$). We consider all values $p=1,\ldots,29$.

Each instance is solved exactly, using the compact formulation for locally budgeted uncertainty. Additionally, we solve each instance using the classic budgeted uncertainty approach, by ignoring the partition into regions and using $\Gamma=\sum_{j\in[K]} \Gamma_j = 10n$. We measure the robust objective value of both solutions with respect to the locally budgeted uncertainty set.

By this setup, we already know that the approach using the locally budgeted uncertainty set must perform better. The question we answer here is how much we lose by ignoring such local information. As discussed in Section~\ref{sec:intro}, local uncertainty naturally arises in some practical applications. Our experiment simulates the effect of using classic budgeted uncertainty in this case.

\subsubsection{Results}

In Figure~\ref{fig1} we show the ratio of average objective values between the solution found by 
the model using classic budgeted uncertainty, and by the
model using locally budgeted uncertainty, for different values of $p$. The higher the ratio, the higher are the additional costs that arise by ignoring the locally budgeted uncertainty structure.

\begin{figure}[htb]
\begin{center}
\includegraphics[width=0.6\textwidth]{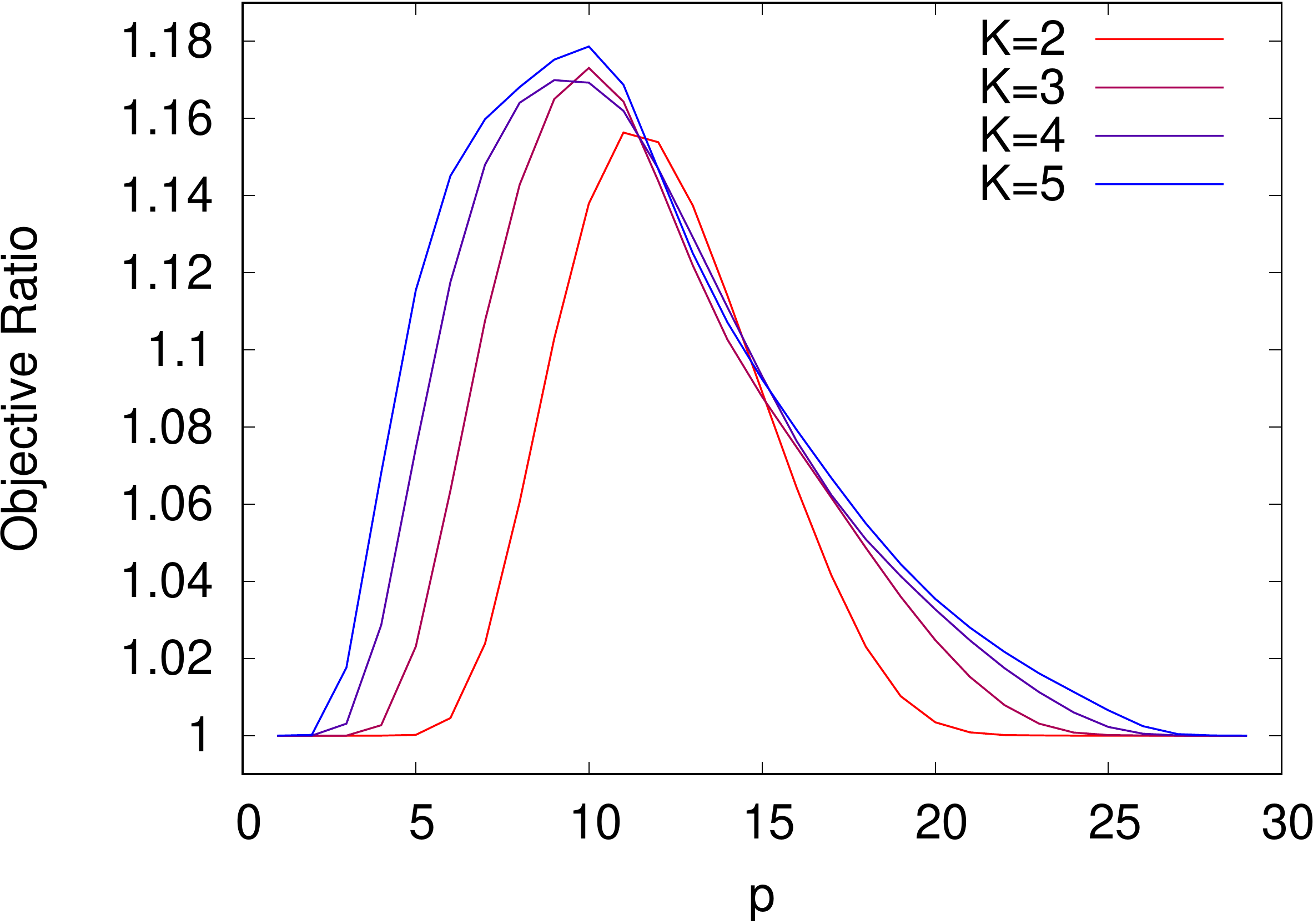}
\caption{Experiment 1: Ratio of between average objective values of solutions based on classic and locally and budgeted uncertainty.}\label{fig1}
\end{center}
\end{figure}

Note that for small ($p=1$) and large ($p=29$) values of $p$, the ratio is close
to one. Hence, the local information does not matter in this setting. The best
choice is to buy the one item $i$ where $\underline{c}_i + d_i$ is smallest (or
to avoid the one item where this value is largest, respectively). For values of
$p$ between these two extremes, solutions differ. The region of values for $p$ where there is a difference between solutions based on classic and locally budgeted uncertainty increases with $K$. For $K=2$ and $p=11$, the average cost difference is $15.6\%$, while this increases to $17.9\%$ for $K=5$ and $p=10$. 

% The larger the value $K$
% (i.e., the more regions exist), the worse becomes the performance of the classic budgeted uncertainty approach in comparison to the locally budgeted approach. The smallest ratio for $K=2$ is $0.823$ and reached for $p=10$, while the smallest ratio for $K=5$ is $0.770$ and reached for $p=8$.

\subsection{Experiment 2}

\subsubsection{Setup}
\label{exp2:setup}

In the previous experiment we considered the effect if the decision maker knows the parameters of a locally budgeted uncertainty set, but chooses to ignore these and use a classic budgeted uncertainty instead. In practice, an uncertainty set is usually not given, but needs to be derived from data.

Hence in this second experiment, we build locally budgeted uncertainty sets in
the same way as before, but then sample $N$ scenarios from the set. To create a
sample scenario $c^k$ with $k\in[N]$, we choose a random value $\gamma_j$ from
$\{0,\ldots,\Gamma_j\}$ uniformly and distribute $\min\{\gamma_j, \sum_{i\in
P_j} d_i\}$ many unit cost increases to all items $i\in P_j$. We do so
iteratively, i.e., we first begin with $c^k=\underline{c}$, and then repeatedly
choose an item from $P_j$ where $c^k_i \le \underline{c}_i + d_i - 1$ at random,
and increase this item's costs by one.

Having constructed $N$ scenarios, we then fit suitable classic and locally budgeted uncertainty sets. 
The focus of this experiment is to derive the regions from the data. We therefore assume that $\underline{c}_i$ and $d_i$ are given for each item. 
% In both cases, we set $\underline{\tilde{c}}_i = \min_{k\in [N]} c^k_i$ and $\tilde{d}_i = \max_{k\in[N]} c^k_i - \underline{\tilde{c}}_i$ as estimates for lower and upper bounds. The classic budgeted uncertainty set then uses $\tilde{\Gamma} = \max_{k\in[N]} \left(\sum_{i\in[n]} c^k_i - \underline{\tilde{c}}_i\right)$ as an estimate for the uncertainty budget. 

To estimate the underlying partition into regions, we can assume
that in a sufficiently large sample $N$, two items from different regions are
not correlated. As each budget constraint applies locally, correlation can only
be found within regions. Based on this idea, we calculate the correlation matrix
using the available sample data. We then consider two items to be connected, if
the absolute value of correlation is above a certain threshold (in this
experiment, we used 0.3). Each connected component then forms its own region.
Note that this way, the number of regions $\tilde{K}$ we use is not prescribed, but estimated from the data.

The classic budgeted uncertainty set uses $\tilde{\Gamma} = \max_{k\in[N]} \left(\sum_{i\in[n]} c^k_i - \underline{c}_i\right)$ as an estimate for the uncertainty budget. For each region, we estimate $\tilde{\Gamma}_j$ in the same way.

We use $n=30$, $p=10$, $K=2,\ldots,5$ and vary the sample size $N$ from 10 to 10,000. As before, for each parameter combination, we construct and solve 10,000 instances.

\subsubsection{Results}

Our results are summarized in Figure~\ref{fig2}. On the horizontal axis, we denote the sample size $N$ (note the logarithmic scale). On the vertical axis, we show average objective values with respect to the original, unknown locally budgeted uncertainty set.

\begin{figure}[htb]
\begin{center}
\includegraphics[width=0.6\textwidth]{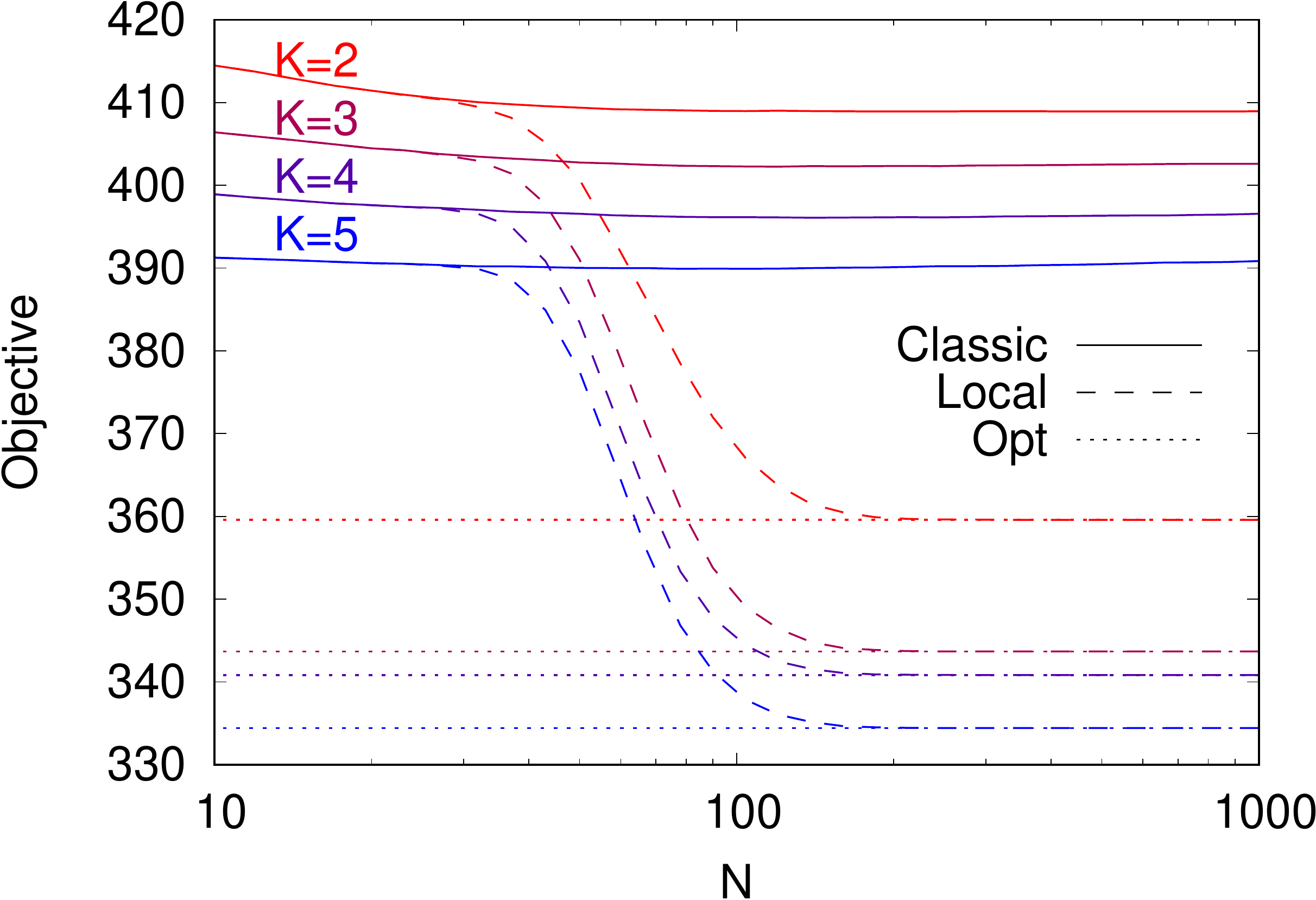}
\caption{Experiment 2: Average objective values of solutions based on locally and classic budgeted uncertainty.}\label{fig2}
\end{center}
\end{figure}

The solid lines indicate the objective value of the solutions based on fitted classic budgeted uncertainty sets, while the dashed lines represent locally budgeted uncertainty sets. The dotted lines indicate the optimal objective value, if the actual uncertainty set were known.

First note that the larger the number of regions $K$, the smaller become objective values overall. For the classic budgeted uncertainty set, the decrease in objective value with increasing sample size $N$ is small, the line is mostly horizontal. This is different for the locally budgeted uncertainty set, where a significant decrease can be observed after the sample size reaches a certain threshold. This begins at around $N=30$, and is completed at approximately $N=110$. We find that even if the locally budgeted uncertainty set is not given explicitly, it is possible to take significant advantage of this model by identifying the corresponding structure in the data.

\subsection{Experiment 3}

\subsubsection{Setup}

While the previous experiments used artificial data that is based on an underlying locally budgeted uncertainty set, we now consider real-world data, where no such underlying structure is known. The data we use was first introduced in \cite{chassein2019algorithms}. It consists of a graph modeling the city of Chicago with 538 nodes and 1308 edges, and 4363 snapshots of traffic speed for each edge over 46 days. Figure~\ref{fig3} shows the structure of the graph.

\begin{figure}[htb]
\begin{center}
\includegraphics[width=0.6\textwidth]{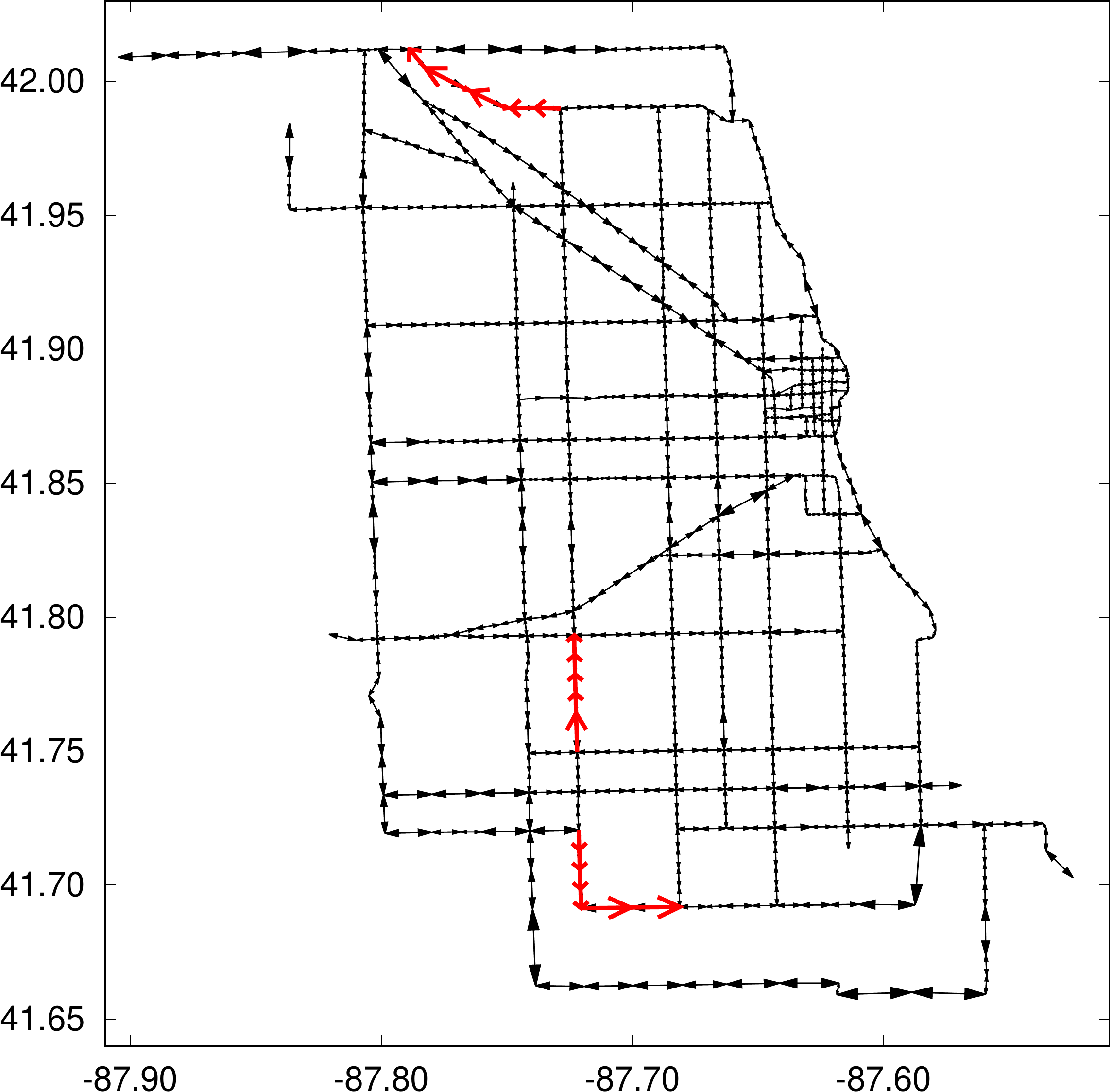}
\caption{Experiment 3: Chicago graph with three regions highlighted.}\label{fig3}
\end{center}
\end{figure}

The data is prepared in the same way as in \cite{chassein2019algorithms}. We use each traffic speed snapshot as a scenario. Of the 4363 scenarios, we use $75\%$ for training our models, and $25\%$ for evaluation. We sample 200 random $s$-$t$ pairs and calculate a shortest path for each pair using each of our models.

The classic budgeted model is trained on the data in the same way as in Experiment~2 (see Section~\ref{exp2:setup}), where $\underline{c}_i$ and $d_i$ are estimated from the data. To model locally budgeted uncertainty sets, we create regions by using edge sequences between any two crossings in one direction. In Figure~\ref{fig3}, we show three such regions in red as an illustration. In total, this results in 546 regions.

We control the degree of conservatism of our two approaches by multiplying the estimated $\tilde{\Gamma}$ value (or $\tilde{\Gamma}_j$ values, respectively) with a budget factor $f$. We use all values of $f$ from 0 to $0.5$ in step size $0.002$.

For each value of $f$ and each model, we solve the 200 shortest path problems and evaluate the path choices in-sample and out-of-sample. We then calculate the average of the average path length and the worst-case path length over the two scenario sets.

\subsubsection{Results}

We show our in-sample results in Figure~\ref{fig4a} and the out-of-sample results in Figure~\ref{fig4b}. On the horizontal axis is the average travel time (in minutes), and the vertical axis is the average worst-case travel time. The results of both models are shown as a line, starting with $f=0$ in the top left, and moving to the right with increasing value of $f$.

\begin{figure}[htb]
\begin{subfigure}[c]{0.49\textwidth}
\includegraphics[width=\textwidth]{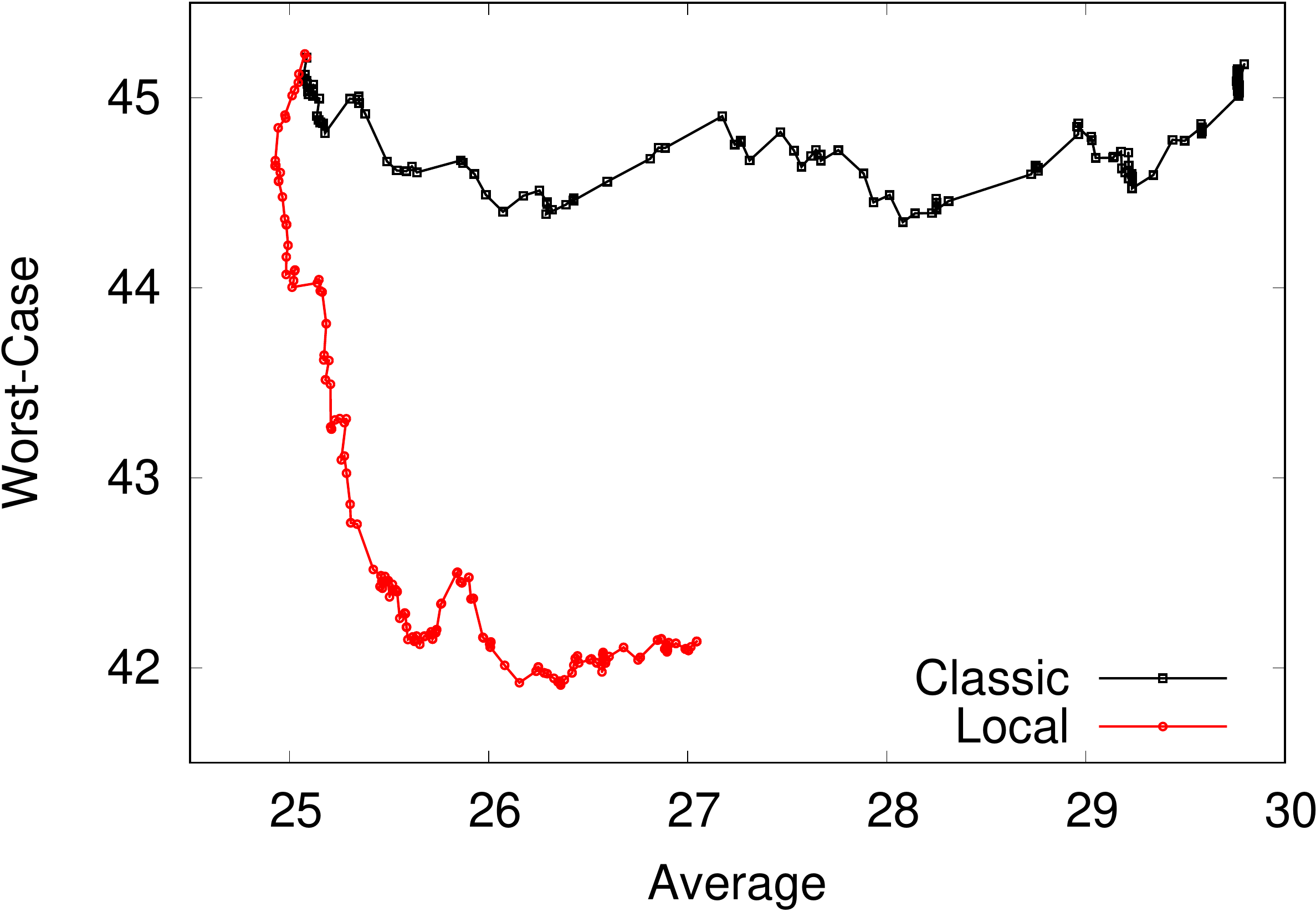}
\subcaption{In-sample results.\label{fig4a}}

\end{subfigure}
\begin{subfigure}[c]{0.49\textwidth}
\includegraphics[width=\textwidth]{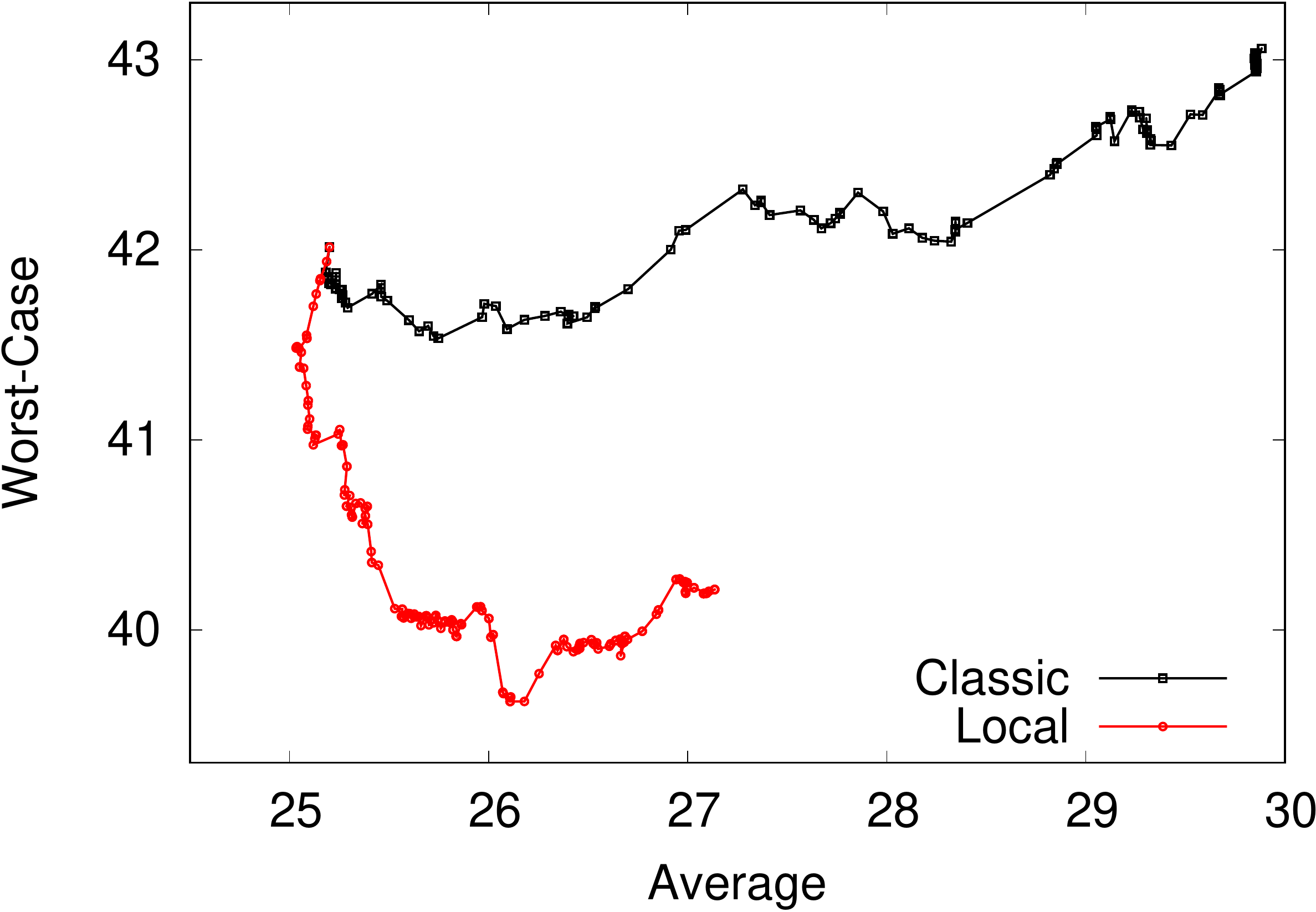}
\subcaption{Out-of-sample results.\label{fig4b}}
\end{subfigure}
\caption{Experiment 3: Average versus average worst-case objective value for different budget factors.}\label{fig4}
\end{figure}

Note that for $f=0$, the classic and the locally budgeted approach result in the same solution, that is, they only optimize for best-case travel times. In an ideal trade-off between average and worst-case travel time, we would expect the lines to reach from the top left corner (low average time, high worst-case time) to the bottom right corner (high average time, low worst-case time).

In Figure~\ref{fig4a} we can see that for the classic approach, no such trade-off can be reached. With increasing budget factor, we increase the average travel time, but do not decrease the worst-case travel times. From the perspective of Pareto optimality, most of the budget factors result in dominated solutions. The locally budgeted uncertainty set, on the other hand, gives a trade-off with increasing budget factor and considerably outperforms the solutions found by the locally budgeted uncertainty set. In the out-of-sample results (Figure~\ref{fig4b}), the classic approach performs even worse, with the curve leading upwards to the top right. The locally budgeted solutions retain a trade-off between average and worst-case time.

Overall, we see that it is possible to model the discrete real-world scenarios more accurately using 
the locally budgeted uncertainty approach, while with the classic budgeted approach, it is not possible to capture the underlying data.

\section{Conclusions and Further Research}
\label{sec:conclusions}

In this paper we introduced a new generalization of budgeted uncertainty sets, where there is a separate uncertainty budget for different regions of items. We showed that for constant number of regions $K$, the robust counterpart remains polynomially solvable if the nominal problem is solvable in polynomial time. For unbounded values of $K$, we show that the robust selection problem can still be solved in polynomial time, while this is not the case for the representative selection problem, even if only one item is chosen from each partition. This extends to other combinatorial problems that include the representative selection problem as a special case. Table~\ref{tab:results} gives an overview to these results. In addition, we show that no parameterized algorithms with running time in $O^*(2^{o(K)})$ exist. To the best of our knowledge, robust optimization problems have not been considered from the perspective of fixed parameter tractability so far.

\begin{table}[htb]
\begin{center}
\begin{tabular}{r|cc}
Problem & $K = O(1)$ & $K=O(n)$ \\
\hline
Unconstrained & P & P \\
Selection & P & P \\
Repr. Selection & P & strongly NPH \\
Spanning Tree & P & strongly NPH \\
$s$-$t$-min-cut & P & strongly NPH \\
Shortest Path & P & strongly NPH 
\end{tabular}
\caption{Overview of complexity results from this paper.}\label{tab:results}
\end{center}
\end{table}

In computational experiments we showed that more general, locally budgeted uncertainty sets can result in better solutions than their classic counterparts using real-world data sets.

Different types of classic budgeted uncertainty sets have been considered in the literature. Instead of the definition used in this paper, where
\[ \cU = \left\{ \pmb{c} = \underline{\pmb{c}} + \pmb{\delta} : \delta_i\in[0,d_i]\ \forall i\in[n],\ \sum_{i\in[n]} \delta_ i \le \Gamma \right\}, \]
it is possible to consider the variant
\[ \cU = \left\{ \pmb{c} : c_i = \underline{c}_i + d_i z_i, z_i\in[0,1]\ \forall i\in[n],\ \sum_{i\in[n]} z_ i \le \Gamma \right\}. \]
All of our theoretical results can be extended to this case. In Theorem~\ref{decomp-theorem}, this means we need to consider $n^K$ instead of $2^K$ subproblems. The hardness results hold as well, as we used $d_i=1$ in the uncertainty sets we constructed.

Furthermore, the locally budgeted uncertainty set proposed in this paper can be seen in the light of data-driven robust optimization (see, e.g., \cite{bertsimas2018data,chassein2019algorithms}), where the aim is to find the most suitable uncertainty set to describe given data. Using local budgets extends the capabilities of classic budgeted uncertainty models, and thus gives more degrees of freedom to describe the data.

From a theoretical perspective, further investigations into the parameterized running time of our
meta-algorithm in Section~\ref{sec:constant} are of interest. 
Note that a minor modification of the proof of Theorem~\ref{thm:ethhard} implies that
no $O^{*}((\sqrt{2}-\epsilon)^{K})$ algorithm for the robust representative selection problem with
locally budgeted uncertainty exists, unless the strong exponential time hypothesis (SETH) fails.
We conjecture that there are combinatorial optimization problems for which the constant $2$ can be improved. Whether this can be done in a meta-algorithm or only for specific combinatorial optimization problems is another interesting open problem.

For the second variant of locally budgeted uncertainty, mentioned in Section~\ref{sec:conclusions},
our gap between positive and negative results is even larger. It is of major interest whether a fixed-parameter
tractable algorithm also exist for this case, or if this slight change in the definition of 
the uncertainty set leads to W[1]-hardness.

\end{document}